\documentclass[11pt,a4paper, reqno]{amsart}

\usepackage{a4wide}

\usepackage{amsfonts,amsthm,amssymb,amsmath,amscd,amsopn,mathabx,mathrsfs}
\usepackage[pdftex]{graphicx}
\usepackage{graphics}
\usepackage[matrix,arrow]{xy}
\usepackage[active]{srcltx}
\usepackage[draft]{todonotes}
\usepackage{tensor}

\usepackage[colorlinks=true]{hyperref}
\hypersetup{
colorlinks   = true,
citecolor    = blue
}

\usepackage{enumitem, hyperref}\hypersetup{colorlinks}

\makeatletter
\def\reflb#1#2{\begingroup
    #2%
    \def\@currentlabel{#2}%
    \phantomsection\label{#1}\endgroup
}
\makeatother

\usepackage{color}

\definecolor{darkred}{rgb}{1,0,0} 
\definecolor{darkgreen}{rgb}{0,0.8,0}
\definecolor{darkblue}{rgb}{0,0,1}

\hypersetup{colorlinks,
linkcolor=darkblue,
filecolor=darkgreen,
urlcolor=darkred,
citecolor=darkblue}

\newtheorem{thm}{Theorem}
\numberwithin{thm}{section}
\numberwithin{equation}{section}
\newtheorem{theorem}[thm]{Theorem}
\newtheorem*{theorem*}{Theorem}

\newtheorem*{corollary*}{Corollary}

\newtheorem{proposition}[thm]{Proposition}

\newtheorem*{conjecture*}{Conjecture}

\newtheorem*{question*}{Question}
\newtheorem{definition}[thm]{Definition}
\newtheorem*{definition*}{Definition}

\newtheorem*{definitions*}{Definitions}

\newtheorem*{rem*}{Remark}
\theoremstyle{remark}
\newtheorem{remark}[thm]{Remark}

\newtheorem*{remark*}{Remark}

\newtheorem*{remarks*}{Remarks}
\newtheorem*{example*}{Example}
\newtheorem{example}[thm]{Example}
\newtheorem*{examples*}{Examples}

\newcommand{\R}{\mathbb{R}}
\newcommand{\Z}{\mathbb{Z}}
\newcommand{\Q}{\mathbb{Q}}
\newcommand{\C}{\mathbb{C}}
\newcommand{\N}{\mathbb{N}}
\newcommand{\T}{\mathbb{T}}

\def\RP{{\mathbb R}P}

\def\PP{{\mathbb P}}
\def\U{\operatorname{U}}

\newcommand{\ep}{\epsilon}
\newcommand{\ga}{\gamma}
\newcommand{\Ga}{\Gamma}

\newcommand{\om}{\omega}

\newcommand{\Sp}{\mathrm{Sp}}

\def\cz{{\mu}}
\def\P{{\mathcal P}}
\def\PP{{\mathscr P}}
\def\HC{{\mathrm{HC}}}
\def\SH{{\mathrm{SH}}}
\def\CC{{\mathrm{CC}}}

\def\H{{\mathrm{H}}}

\def\bga{{\bar\gamma}}
\def\S{{\Sigma}}
\def\L{L^{2n+1}_p(\ell_0,\dots,\ell_n)}
\def\balpha{{\bar\alpha}}
\def\A{{\mathcal{A}}}
\def\G{{\mathcal{G}}}
\def\age{{\text{age}}}
\def\bx{{\bar x}}
\def\W{{\widehat W}}
\newcommand{\veps}{\varepsilon}
\def\Im{{\text{Im}}}
\def\tga{{\widetilde\gamma}}
\def\bga{{\bar\gamma}}

\begin{document}

\title[Symmetric periodic Reeb orbits on the sphere]{Symmetric periodic Reeb orbits on the sphere}

\author[Miguel Abreu]{Miguel Abreu}
\author[Hui Liu]{Hui Liu}
\author[Leonardo Macarini]{Leonardo Macarini}

\address{Center for Mathematical Analysis, Geometry and Dynamical Systems,
Instituto Superior T\'ecnico, Universidade de Lisboa, 
Av. Rovisco Pais, 1049-001 Lisboa, Portugal}
\email{mabreu@math.tecnico.ulisboa.pt}

\address{School of Mathematics and Statistics, Wuhan University, Wuhan 430072, Hubei, China}
\email{huiliu00031514@whu.edu.cn}

\address{IMPA,
Estrada Dona Castorina, 110, Rio de Janeiro, 22460-320, Brazil (on leave from Instituto Superior T\'ecnico, University of Lisbon)}
\email{leonardo@impa.br}


\thanks{MA and LM were partially funded by FCT/Portugal through UID/MAT/04459/2020 and PTDC/MAT-PUR/29447/2017. LM was also partially funded by CNPq, Brazil. HL was partially supported by NSFC (Nos. 12371195, 12022111) and the Fundamental Research Funds for the Central Universities (No. 2042023kf0207).}

\begin{abstract}
A long standing conjecture in Hamiltonian Dynamics states that every contact form on the standard contact sphere $S^{2n+1}$ has at least $n+1$ simple periodic Reeb orbits. In this work, we consider a refinement of this problem when the contact form has a suitable symmetry and we ask if there are at least $n+1$ simple symmetric periodic orbits. We show that there is at least one symmetric periodic orbit for any contact form and at least two symmetric closed orbits whenever the contact form is dynamically convex.
\end{abstract}

\maketitle

\tableofcontents

\section{Introduction}

Consider the unit sphere $S^{2n+1}=\{v \in \R^{2n+2};\,\|v\|=1\}$ and the Liouville form $\lambda=\frac 12 \sum_{i=1}^{n+1} q_idp_i - p_idq_i$. The standard contact structure on $S^{2n+1}$ is given by $\xi=\ker \lambda|_{S^{2n+1}}$ and a contact form supporting this (cooriented) structure is a 1-form $\alpha=f\lambda|_{S^{2n+1}}$, where $f: S^{2n+1} \to \R$ is a positive function. Associated to $\alpha$ we have its Reeb vector field $R_\alpha$ uniquely characterized by the equations $\iota_{R_\alpha} d\alpha=0$ and $\alpha(R_\alpha)=1$. Reeb flows form a prominent class of Hamiltonian systems on energy levels and the study of Reeb flows on the standard contact sphere $S^{2n+1}$ is equivalent to the study of Hamiltonian flows of proper homogeneous of degree two Hamiltonians $H: \R^{2n+2} \to \R$.

Let us denote by $\P$ the set of simple (i.e., non-iterated) closed Reeb orbits of $\alpha$. A long standing conjecture in Hamiltonian Dynamics establishes that for every contact form $\alpha$ on $S^{2n+1}$ we have that $\#\P \geq n+1$. It was proved for $n=1$ by Cristofaro-Gardiner and Hutchings \cite{CGH} (in the more general setting of Reeb flows in dimension three) and independently by Ginzburg, Hein, Hryniewicz and Macarini \cite{GHHM}; see also \cite{LL} where an alternate proof was given using a result from \cite{GHHM}. In higher dimensions, the conjecture is completely open without additional assumptions on $\alpha$, such as convexity or certain index requirements or non-degeneracy of closed Reeb orbits. For instance, it is not know if every contact form on $S^5$ has at least two simple closed Reeb orbits (although the conjecture states that we should have at least three closed orbits).

In order to explain the convexity assumption, note that there is a natural bijection between contact forms $\alpha$ on $(S^{2n+1},\xi)$ and starshaped hypersurfaces $\S_\alpha$ in $\R^{2n+2}$  given by
\[
\alpha=f\lambda|_{S^{2n+1}} \longleftrightarrow\Sigma_\alpha=\{\sqrt{f(x)}x;\, x\in S^{2n+1}\}
\]
and it satisfies the property that if $\Sigma_\alpha$ is the energy level of a homogeneous of degree two Hamiltonian $H_\alpha: \R^{2n+2} \to \R$ then the Hamiltonian flow of $H_\alpha$ on $\Sigma_\alpha$ is equivalent to the Reeb flow of $\alpha$. We say that $\alpha$ is \emph{convex} if $\Sigma_{\alpha}$ bounds a strictly convex domain.

When $\alpha$ is convex, Ekeland and Hofer proved in \cite{EH} that $\#\P \geq 2$. Later on, Long and Zhu in the remarkable work \cite{LZ} showed that $\#\P\geq \lfloor (n+1)/2 \rfloor +1$. This result was improved when $n$ is even by Wang \cite{Wa}, furnishing the lower bound $\#\P\geq \lceil (n+1)/2 \rceil +1$. In particular, $\#\P \geq n+1$ for every convex contact form on $S^{2n+1}$ when $n=2$ (it was proved before in \cite{HLW}). In \cite{Wa2}, Wang proved that it also holds when $n=3$.

The notion of convexity is not natural from the point of view of Symplectic Topology since it is not a condition invariant under contactomorphisms. An alternative definition was introduced by Hofer, Wysocki and Zehnder \cite{HWZ}, called \emph{dynamical convexity}. A contact form on $S^{2n+1}$ is dynamically convex if the Conley-Zehnder index of every closed orbit is bigger than or equal to $n+2$. It is not hard to see that every convex contact form is dynamically convex. Clearly, dynamical convexity is invariant under contactomorphisms. Moreover, dynamical convexity is more general than convexity: there are contact forms that are dynamically convex and are not contactomorphic to a convex one \cite{CE1,CE2}; see also \cite{AM,GM}.

When $\alpha$ is dynamically convex, Ginzburg and G\"urel \cite{GG} and, independently, Duan and Liu \cite{DL}, proved that $\#\P\geq \lceil (n+1)/2 \rceil +1$, generalizing the previous works of Long-Zhu and Wang.

In this work, we will consider a refinement of this conjecture when the contact form has a symmetry. More precisely, given an integer $p>0$, consider the $\Z_p$-action on $S^{2n+1}$, regarded as a subset of $\C^{n+1}\setminus \{0\}$, generated by the map
\begin{equation}
\label{eq:Z_p-action}
\psi(z_0, \dots, z_n) = \left(e^{\frac{2\pi\sqrt{-1}\ell_0}{p}}z_0, e^{\frac{2\pi\sqrt{-1}\ell_1}{p}}z_1, \dots, e^{\frac{2\pi\sqrt{-1}\ell_n}{p}}z_n \right),
\end{equation}
where $\ell_0, \ldots, \ell_n$ are integers called the \emph{weights} of the action. Such an action is free when the weights  are coprime with $p$ and in that case we have a lens space obtained as the quotient of $S^{2n+1}$ by the action of $\Z_p$. We denote such a lens space by $L_p^{2n+1}(\ell_0, \ell_1, \ldots, \ell_n)$. From now on, we will assume that the weights are coprime with $p$.

The question that we will address in this work is the following one. Given the action of a Lie group $G$ on a manifold $M$ and a vector field $X$ on $M$ invariant under this action, we say that a periodic orbit $\ga$ of $X$ is \emph{symmetric} if $g(\text{Im}(\ga))=\text{Im}(\ga)$ for every $g \in G$.

\vskip .2cm
\noindent
{\bf Question:} Let $\alpha$ be a contact form on $S^{2n+1}$ invariant under the above $\Z_p$-action generated by $\psi$. Denote by $\P_s$ the set of simple \emph{symmetric} closed orbits of the Reeb flow of $\alpha$. Is it true that $\#\P_s \geq n+1$?
\vskip .2cm

Note that when $p=1$ the $\Z_p$-action is trivial and therefore every periodic orbit is symmetric. Thus, in this case the question is equivalent to the aforementioned problem if every contact form on $S^{2n+1}$ carries at least $n+1$ closed orbits. Moreover, the irrational ellipsoids are invariant under this $\Z_p$-action for any $p$ and weights $\ell_0, \ldots, \ell_n$ and carry precisely $n+1$ periodic orbits which are all symmetric.

Girardi proved in \cite{Gi} that when $p=2$, that is, when $\psi$ is the antipodal map, we have that $\#\P_s \geq 1$ for any symmetric contact form. In the same work, it was proved that if $\Sigma_\alpha$ satisfies suitable pinching conditions then $\#\P_s \geq n+1$ when $p=2$. This result was generalized by Abreu and Macarini in \cite[Theorem 1.21]{AM} for any $p$ assuming that $\alpha$ is convex and $\ell_0=\dots=\ell_n=1$. In \cite{Bah}, B\"ahni showed that $\#\P_s \geq 1$ for any symmetric contact form when $p$ is even.

When $\alpha$ is convex and symmetric, for any $p$ and weights $\ell_0,\dots,\ell_n$, it was proved by Zhang \cite{Zha} that $\#\P_s \geq 2$. This result was generalized to dynamically convex contact forms by Liu and Zhang \cite{LiZ} assuming that $p=2$.

The main results in this work are the following generalizations of these results.

\begin{theorem}
\label{thm:existence}
Let $\alpha$ be any contact form on $S^{2n+1}$ invariant under the above $\Z_p$-action. Then $\alpha$ has at least one symmetric closed orbit.
\end{theorem}

This result generalizes the results of Rabinowitz \cite{Ra} (that corresponds to the case $p=1$), Girardi \cite{Gi} (that corresponds to the case $p=2$) and B\"ahni \cite{Bah} (that corresponds to the case that $p$ is even).

\begin{theorem}
\label{thm:multiplicity}
Let $\alpha$ be a dynamically convex contact form on $S^{2n+1}$ invariant under the above $\Z_p$-action. Then $\alpha$ has at least two simple symmetric closed orbits.
\end{theorem}

The last result generalizes the results of Zhang \cite{Zha} (that corresponds to the case that $\alpha$ is convex) and Liu-Zhang \cite{LiZ} (that corresponds to the case $p=2$).

Let $\beta$ be the induced contact form on $\L$. A natural question is how the previous results can be seen in terms of periodic orbits of $\beta$. The simple unparametrized symmetric closed orbits of $\alpha$ are in bijection with the simple unparametrized closed orbits of $\beta$ whose homotopy classes are generators of $\pi_1(\L)$; see Proposition \ref{prop:symmetric orbit}. Thus, Theorem \ref{thm:existence} is equivalent to the statement that given any $a \in \pi_1(\L)$ then every contact form $\beta$ on $\L$ has at least one periodic orbit with homotopy class $a$. Theorem \ref{thm:multiplicity}, in turn, is equivalent to the statement that given any $a \in \pi_1(\L)$ and a dynamically convex contact form $\beta$ on $\L$ (we say that $\beta$ is dynamically convex if so is its lift to $S^{2n+1}$) we have at least two \emph{$a$-simple} periodic orbits, where a periodic orbit is called $a$-simple if it has homotopy class $a$ and it is not an iterate of another periodic orbit with homotopy class $a$, although it can be the iterate of a closed orbit with another homotopy class.
 
\vskip .3cm
\noindent {\bf Organization of the paper. }
The rest of the paper is organized as follows. The background on symplectic homology and Lusternik-Schnirelmann theory in Floer homology necessary for this work is presented in Section \ref{sec:SH&LS}. Symplectic homology/cohomology of orbifolds is discussed in Section \ref{sec:SHO}, where there is also a brief explanation of Chen-Ruan cohomology in Section \ref{sec:CR}. The proofs of Theorems \ref{thm:existence} and \ref{thm:multiplicity} are given in Sections \ref{sec:proof existence} and \ref{sec:proof multiplicity} respectively.
 
\vskip .2cm
\noindent {\bf Acknowledgements.}
The third author is grateful to Alexandru Oancea and Zhengyi Zhou for useful discussions regarding this paper. We are grateful to a referee for his/her helpful suggestions.

\section{Symplectic homology and Lusternik-Schnirelmann theory}
\label{sec:SH&LS}

\subsection{Sign conventions}
\label{sec:signs}

Our sign conventions agree with the ones in \cite{GG}. In what follows, we will explain these conventions and compare them with the ones in \cite{GZ}. It will be important to relate, in Section \ref{sec:sh-orb}, the equivariant symplectic homology in \cite{GG} with the equivariant symplectic \emph{co}homology in \cite{GZ}.

Given an exact symplectic manifold $(W,\om=d\lambda)$, a Hamiltonian $H_t: W \to \R$ and a closed curve $\ga: S^1 \to W$ we take the action functional
\[
A_H(\ga) =  \int_\ga \lambda - \int_0^1 H_t(\ga(t))\,dt
\]
as in \cite{GG}. This is \emph{minus} the action functional in \cite{GZ}. Thus, in both works Hamilton's equation, whose closed orbits are critical points of the action functional,  is $i_{X_{H_t}}\om=-dH_t$. The Conley-Zehnder index $\cz$ is normalized so that, when $S$ is a small positive definite symmetric matrix, the symplectic path $\Ga: [0,1] \to \Sp(2n)$ given by $\Ga(t)=\exp(tJ_0S)$ has $\cz(\Ga)=n$, where $J_0$ is the complex structure such that $\om_0(\cdot,J_0\cdot)$ is a metric, with $\om_0=\sum_i dq_i \wedge dp_i$ being the canonical symplectic form in $\R^{2n}$. It agrees with the conventions in \cite{GG} and \cite{GZ}.

As in \cite{GG}, a compatible almost complex structure $J$ on $W$ is defined by the condition that $\om(J\cdot,\cdot)$  (with $J$ acting in the \emph{first} argument) is a Riemannian metric. It has the opposite sign of the complex structure in \cite{GZ}, determined by the condition that $\om(\cdot,J\cdot)$ (with $J$ acting in the \emph{second} argument) is a Riemannian metric. Floer's equation is given by
\[
\partial_s u + J(\partial_t u - X_{H_t}) = 0
\]
as in \cite{GG}. Floer's equation in \cite{GZ} looks the same, but, since the complex structures in \cite{GG} and \cite{GZ} have opposite signs, solutions $u(s,\cdot)$ of our Floer's equation correspond to solutions $u(-s,\cdot)$ of Floer's equation in \cite{GZ}.

\subsection{Symplectic homology}
\label{sec:SH}

Let $(M^{2n+1},\xi)$ be a closed contact manifold endowed with a strong symplectic filling given by a Liouville domain $W$ such that $c_1(TW)|_{H_2(W,\R)}=0$. The symplectic homology $\SH_*(W)$ of $W$ is a symplectic invariant introduced by Cieliebak, Floer, Hofer and Viterbo \cite{CFH,Vit}. It is obtained as a direct limit of Floer homologies, graded by the Conley-Zehnder index, of a suitable class of Hamiltonians on $W$ and a nice reference with full details can be found in \cite{BO09b}.

Symplectic homology has a natural action filtration which allows us to define the \emph{negative} and \emph{positive} symplectic homologies $\SH^-_*(W)$ and $\SH^+_*(W)$. These homologies are related via the tautological exact triangle
\begin{equation}
\label{eq:etsh}
\xymatrix{
		\SH^-_*(W) \ar[rr] & &\SH_*(W)\ar[ld] \\
		& \SH^+_*(W) \ar[lu]^{[-1]} &}.
\end{equation}

Symplectic homology has a decomposition given by the free homotopy classes $a$ of loops in $W$ and the previous exact triangle respects this decomposition. Denote by $\SH_*^a(W)$ and $\SH^{a,\pm}_*(W)$ the corresponding homologies. The negative symplectic homology is generated by constant periodic orbits and therefore $\SH_*^-(W) = \SH^{0,-}_*(W)$. It turns out that $\SH_*^-(W)$ is isomorphic to $\H^{(n+1)-*}(W)$.

The action functional of an autonomous Hamiltonian is invariant under the natural $S^1$-action given by translations in the parametrization of the closed curves. An equivariant version of symplectic homology, taking into account this symmetry, is the equivariant symplectic homology $\SH^{S^1}_*(W)$ introduced by Viterbo \cite{Vit} and developed by Bourgeois and Oancea \cite{BO10, BO13a, BO13b, BO17}. As established in \cite{BO13b}, the groups $\SH_{*}(W)$ and $\SH^{S^1}_*(W)$ fit into the Gysin exact triangle
\begin{equation}
\label{eq:et gysin}
\xymatrix{
		\SH_{*}(W)\ar[rr] & &\SH^{S^1}_*(W)\ar[ld]^D \\
		& \SH^{S^1}_{*-2}(W) \ar[lu]^{[1]} &}
\end{equation}
where $D$ is the so called \emph{shift operator} \cite{GG}.

As in the non-equivariant case, the equivariant symplectic homology has an action filtration that allows us to define the positive/negative equivariant symplectic homology $\SH^{S^1,\pm}_*(W)$. It turns out that the positive equivariant symplectic homology with rational coefficients can be obtained as the homology of a chain complex $\CC_*(\alpha)$ generated by the good closed Reeb orbits of a non-degenerate contact form $\alpha$ on $M$; cf. \cite[Proposition 3.3]{GG} and \cite[Lemma 2.1]{GU}. This complex is filtered by the action and graded by the Conley-Zehnder index. In this work, we will take the grading using trivializations of the contact structure along periodic orbits $\ga$ using sections of the determinant line bundle $\Lambda_\C^{n+1}TW$ following \cite{McL,Sei}. These trivializations behave well under iterations:  the trivialization induced on $\gamma^j$ coincides, up to homotopy, with the $j$-th iterate of the  trivialization over $\gamma$; see \cite[Section 2.3]{AM}. If $c_1(TW)$ does not vanish, this grading is, in general, fractional; cf. \cite[Section 2.3]{AM}. The differential in the complex $\CC_*(\alpha)$, but not its homology, depends on several auxiliary choices, and the nature of the differential is not essential for our purposes. The complex $\CC_*(\alpha)$ is functorial in $\alpha$ in the sense that a symplectic cobordism equipped with a suitable extra structure gives rise to a map of complexes that induce an isomorphism in the homology. For the sake of brevity and to emphasize the obvious analogy with contact homology, we denote the homology of $\CC_*(\alpha)$ by $\HC_*(W)$ rather than $\SH^{S^1,+}_*(W)$. Furthermore, once we fix a free homotopy class of loops in $W$, the part of $\CC_*(\alpha)$ generated by good closed Reeb orbits in that class is a subcomplex. As a consequence, the entire complex $\CC_*(\alpha)$ breaks down into a direct sum of such subcomplexes indexed by free homotopy classes of loops in $W$.

A remarkable observation by Bourgeois and Oancea in \cite[Section 4.1.2]{BO17} is that under suitable additional assumptions on the indices of closed Reeb orbits the positive equivariant symplectic homology is defined even when $M$ does not have a symplectic filling, using the symplectization of $M$, and therefore is a contact invariant. To be more precise, we assume that $c_1(\xi)|_{H_2(M,\R)}=0$ and that $M$ admits a non-degenerate contact form $\alpha$ such that all of its closed \emph{contractible} Reeb orbits have Conley--Zehnder index strictly greater than $3-n$. Furthermore, under this assumption once again the equivariant symplectic homology of $M$ with rational coefficients can be described as the homology of a complex $\CC_*(\alpha)$ generated by good closed Reeb orbits of $\alpha$, graded by the Conley-Zehnder index and filtered by the action. The complex breaks down into the direct sum of subcomplexes indexed by free homotopy classes of loops \emph{in $M$}. We will use the notation $\HC_*^a(M)$ to denote the homology of the complex generated by the orbits with free homotopy class $a$.

With the same assumption on $\alpha$, we can also define a \emph{non-equivariant} symplectic homology $\SH_*(M)$ in terms of the symplectization of $M$ and that therefore does not require the existence of a filling $W$ and is a contact invariant. The construction of this homology is analogous to the one in \cite[Section 4.1.2]{BO17}: the admissible Hamiltonians $H^c: (0,\infty) \times M \to \R$ are of the form $H^c(r,x)=h^c(r)$, with $h^c: (0,\infty) \to \R$ a convex increasing function such that $dh^c(r)/dr \to 0$, $r \to 0$, $dh^c(r)/dr = c$ for $r \gg 0$ and $d^2h^c(r)/dr^2>0$ whenever $dh^c(r)/dr < c$. (Here, of course, $H^c$ has to be slightly perturbed to a time-dependent Hamiltonian in order to achieve nondegeneracy of its 1-periodic orbits.) The condition on the indices of the orbits prevents the bubbling off of planes in the concave end, so that the resulting Floer homology is defined. Then the argument developed in \cite{BO13b} to establish the Gysin exact triangle \eqref{eq:et gysin} works verbatim to prove the exact triangle
\begin{equation}
\label{eq:et gysin symplectization}
\xymatrix{
		\SH_{*}(M)\ar[rr] & &\HC_*(M)\ar[ld]^D \\
		& \HC_{*-2}(M) \ar[lu]^{[1]} &}.
\end{equation}
These groups have a decomposition by the free homotopy classes in $M$ and this exact triangle respects this decomposition.

Finally, let us briefly recall the definition of equivariant local symplectic homology. Given an isolated (possibly degenerate) closed orbit $\ga$ of $\alpha$ we have its equivariant local symplectic homology $\HC_*(\ga)$ \cite{GG,HM}. It is supported in $[\cz(\ga),\cz(\ga)+\nu(\ga)]$ (i.e. $\HC_k(\ga)=0$ for every $k \notin [\cz(\ga),\cz(\ga)+\nu(\ga)]$), where $\cz(\ga)$ is the Conley-Zehnder index of $\ga$ and $\nu(\ga)$ is the nullity of $\ga$, i.e., the geometric multiplicity of the eigenvalue 1 of the linearized Poincar\'e map. (The index of a degenerate periodic orbit here is defined as the lower semicontinuous extension of the Conley-Zehnder index for non-degenerate periodic orbits; cf. \cite{GM}.) If $\alpha$ carries finitely many simple closed orbits with free homotopy class $a$ then if $\HC^a_k(M) \neq 0$ there exists a closed orbit $\ga$ with free homotopy class $a$ such that $\HC_k(\ga)\neq 0$; see \cite{HM}. In particular, this implies that $|\cz(\ga)-k|\leq 2n$.

\subsection{Lusternik-Schnirelmann theory}
\label{sec:LS}

In what follows, we will explain briefly the results from Lusternik-Schnirelmann theory in Floer homology necessary for this work. We refer to \cite{GG} for details.

Let $(M^{2n+1},\xi)$ be a closed contact manifold and $a$ a free homotopy class in $M$. Let $\alpha$ be a non-degenerate contact form on $M$ such that every contractible closed Reeb orbit has index strictly bigger than $3-n$. Given numbers $0<T_1<T_2\leq \infty$ we denote by $\HC^{a,(T_1,T_2)}_*(\alpha)$ the equivariant symplectic homology of $\alpha$ with free homotopy class $a$ and action window $(T_1,T_2)$. When $\alpha$ is degenerate and both $T_1$ and $T_2$ are not in the action spectrum of $\alpha$ we define $\HC^{a,(T_1,T_2)}_*(\alpha)$ as $\HC_*^{a,(T_1,T_2)}(\balpha)$ for some small non-degenerate perturbation $\balpha$ of $\alpha$. (Recall that the action spectrum of $\alpha$ is given by $\A(\alpha)=\{\int_\ga\alpha;\,\ga\ \text{is a closed Reeb orbit of}\ \alpha\}$.) Given $T \in (0,\infty]$ we denote by $\HC_*^{a,T}(\alpha)$ the filtered equivariant symplectic homology $\HC_*^{a,(\ep,T)}(\alpha)$ for some $\ep>0$ sufficiently small such that $\ep<\min\{T;\,T \in \A(\alpha)\}$.

Let $\alpha$ be a contact form on $M$ such that every closed orbit with homotopy class $a$ is isolated and $T \in (0,\infty]\setminus\A(\alpha)$. Given a non-trivial element $w \in \HC^{a,T}_k(\alpha)$ we have a spectral invariant given by
\[
c_w(\alpha)=\inf\{T' \in (0,T)\setminus\A(\alpha);\,w\in\text{Im}(i^{a,T'})\}
\]
where $i^{a,T'}: \HC_*^{a,T'}(\alpha) \to  \HC_*^{a,T}(\alpha)$ is the map induced in the homology by the inclusion of the complexes. It turns out that there exists a periodic orbit $\ga$ with action $c_w(\alpha)$ and free homotopy class $a$ such that $\HC_k(\ga)\neq 0$; cf. \cite[Corollary 3.9]{GG}. 

The shift operator $D: \HC^{a}_*(M) \to  \HC^{a}_{*-2}(M)$ in the exact triangle \eqref{eq:et gysin symplectization} satisfies the property
\begin{equation}
\label{eq:shift op}
c_w(\alpha) > c_{D(w)}(\alpha),
\end{equation}
see \cite[Theorem 1.1]{GG}. Suppose now that there exists $l_a \in \Q$ such that $\HC_{l_a+2k}^a(M) \neq 0$ and $D: \HC_{l_a+2k}^a(M) \to  \HC^{a}_{l_a+2k-2}(M)$ is an isomorphism for every $k \in \N_0$. Then it follows from \eqref{eq:shift op} and the previous discussion that there exists an injective map $\psi: \N_0 \to \PP_a$, where $\PP_a$ is the set of closed orbits of $\alpha$ with homotopy class $a$, such that $\ga_k:=\psi(k)$ satisfies $\HC_{l_a+2k-2}(\ga_k)\neq 0$. In particular, $|\cz(\ga_k)-(l_a+2k-2)|\leq 2n$. The map $\psi$ is called a \emph{carrier map}; see \cite{GG}.

\section{Symplectic homology/cohomology of orbifolds}
\label{sec:SHO}

Let $(M,\xi)$ be a closed \emph{smooth} contact manifold endowed with an exact \emph{orbifold} symplectic filling $W$. Gironella and Zhou introduced in \cite{GZ} a symplectic \emph{co}homology of $W$ that will be fundamental in this work. Before we discuss this, we have to recall the definition of the Chen-Ruan cohomology.

\subsection{Chen-Ruan cohomology}
\label{sec:CR}

In this section, we will give a brief definition of Chen-Ruan cohomology enough for our purposes. In particular, we will not address its product structure since it will not be necessary here. We refer to \cite{ALR} and \cite{CR} for details.

A \emph{Lie groupoid} $\G$ consists of a manifold of objects $G_0$ and a manifold of arrows $G_1$ with the following structure maps: the source and target maps $s$, $t: G_1 \to G_0$, unit map $u: G_0 \to G_1$, inverse map $i : G_1 \to G_1$ and composition map $m: G_1 \tensor[_s]{\times}{_t} G_1 \to G_1$. All these maps are required to be smooth and to obey the obvious properties: composition has to be associative and has to interact in the expected way with the unit and inverse maps. Moreover we require $s$, $t$ to be submersions; this is needed so that $G_1 \tensor[_s]{\times}{_t} G_1$ is a manifold.

A Lie groupoid $\G$ is said to be proper if the map $(s,t): G_1 \to G_0 \times G_0$ is proper. The Lie groupoid $\G$ is said to be \'etale if $s$ and $t$ are local diffeomorphisms; in this case we can define the dimension of $\G$ to be $\dim \G = \dim G_0 = \dim G_1$. An \emph{orbifold groupoid} is an \'etale and proper Lie groupoid. Associated to it we have an underlying topological space: its orbit space $|\G| = G_0/\simeq$ where $\simeq$ is the equivalence relation defined by $x \simeq y$ if and only if there is $g \in G_1$ such that $x=s(g)$ and $y=t(g)$.

A basic example is given by a \emph{global quotient orbifold}. Let $X$ be a manifold and $G$ a finite group acting on it. The global quotient $[X/G]$ is given by the Lie groupoid $\G = G \ltimes X$ with $G_0=X$, $G_1=G\times X$, $s(g,x)=x$ and $t(g,x)=g(x)$.

Let $\G$ be an orbifold groupoid and consider the commutative diagram
\begin{equation*}
\xymatrix{
S_\G \ar[r] \ar[d]^{\beta} & G_1 \ar[d]^{(s,t)} \\
G_0 \ar[r]^{\Delta} & G_0 \times G_0
}
\end{equation*}
where $\Delta$ is the diagonal map and $S_\G=\{g \in G_1;\,s(g) = t(g)\}$ is naturally equipped with the map $\beta: S_\G \to G_0,\ g \mapsto t(g) = s(g)$. Any $h \in G_1$ induces a diffeomorphism $\beta^{-1}(s(h)) \to \beta^{-1}(t(h))$ via the action by conjugation $h \cdot g = hgh^{-1}$ for $g \in \beta^{-1}(s(h))$. In particular, we can form the Lie groupoid $\G \ltimes S_\G$ with the manifold of objects given by $S_\G$ and manifold of arrows given by $G_1 \tensor[_s]{\times}{_\beta} S_\G = \{(h, g) \in G_1 \times S_\G;\,s(h) = \beta(g)\}$, with source map $s(h, g)=g$ and target map $t(h, g) = hgh^{-1}$. This groupoid $\G \ltimes S_\G$ is called the \emph{inertia grupoid} of $\G$ and denoted by $\Lambda\G$.

\begin{example}
\label{ex:C^n}
Let $G \subset \U(n)$ be a finite subgroup and consider the global quotient orbifold $\G=G\ltimes \C^n$ with $G_0=\C^n$, $G_1=G\times \C^n$, $s(g,x)=x$ and $t(g,x)=g(x)$. Then
\[
|\Lambda\G| = \{(x,(g)_{G_x});\, x \in \C^n\text{ and }g \in G_x\}
\]
where $G_x$ is the isotropy group of $x$ (i.e. the subgroup $\{g \in G;\,g(x)=x\}$) and $(g)_{G_x}$ denotes the conjugacy class of $g$ in $G_x$. If $G$ is abelian and acts freely in $\C^n\setminus\{0\}$ then clearly
\[
|\Lambda\G| \cong\C^n \sqcup \{g \in G;\,g\neq 1\}.
\]								
\end{example}

The ungraded Chen-Ruan cohomology of an orbifold $\G$ is the just the (ungraded) singular cohomology of $|\Lambda\G|$. In order to describe its grading, we have to introduce a decomposition of $\Lambda\G$ into connected components.

For this, we define an equivalence relation $\simeq$ on $S_{\G}$ as follows.  On each local uniformizer $G\ltimes U$ (a local uniformizer is a sort of local chart for orbifolds; see \cite[Definition 2.8]{GZ} for a precise definition), for $x,y\in U$, any two $g_1\in G_x\subset G$ and $g_2\in G_y\subset G$ are equivalent with respect to $\simeq$ if $g_1,g_2$ are conjugated in $G$.  More generally, $g_1\simeq g_2$ iff $g_1,g_2$ are connected by a sequence $\{h_1=g_1,h_2,\ldots,h_{n-1}, h_n=g_2\}\subset S_{\G}$, such that $h_i,h_{i+1}$ are equivalent in a local uniformizer. We denote by $(g)$ the equivalence class of $g\in S_{\G}$ under the just defined relation $\simeq$, and by $T$ the set $S_{\G}/\simeq$ of all such equivalence classes. Then $\Lambda\G$ can be decomposed as
\[
\Lambda\G  = \bigsqcup_{(g)\in T}  \G_{(g)},
\] 
where $\G_{(g)}$ is the $\G$-inertia groupoid on the $(g)$-component of $S_{\G}$. Notice in particular that $\G_{(1)}$ is naturally isomorphic to $\G$. The orbifolds $\G_{(g)}$, for each $(g) \in T$, $g\neq 1$, are called the \emph{twisted sectors} of $\G$ and $\G_{(1)}$ is called the \emph{untwisted sector}.

\begin{example}
If $G \subset \U(n)$ is an abelian subgroup that acts freely in $\C^n\setminus\{0\}$ and $\G$ is the orbifold given in Example \ref{ex:C^n} then clearly $|\Lambda\G_1| = \C^n$ and $|\Lambda\G_g| = \{0\}$ for every  $g \in G$ such that $g\neq 1$.			
\end{example}

Now, assume that $\G$ admits an almost complex structure, i.e., an almost complex structure $J$ on $G_0$ such that $s^*J=t^*J$. Let $g \in S_\G$ be an arrow with $s(g) = t(g) = x$. Choose an orbifold chart $(U,G_x,\phi)$ with $U$ embedded in $G_0$. Then $g \in G_x$ is a map $g: U \to U$. Its differential at $x$ is a map $(dg)_x: T_xG_0 \to T_xG_0$. The almost complex structure $J_x$ endows $T_xG_0$ with the structure of a complex vector space, hence identifying $T_xG_0$ with $\C^n$ where $2n=\dim\G$; the condition that $s^*J = t^*J$ implies that $(dg)_x$ preserves the almost complex structure $J_x$ on $T_xG_0$, so it can be regarded as a linear transformation of complex vector spaces or, equivalently, a matrix in $GL(n,\C)$. Since g has finite order, say $m \in \N$, the eigenvalues of $(dg)_x$, always regarded as a complex linear transformation, are of the form $e^{2\pi i\lambda_1},\dots,e^{2\pi i\lambda_n}$ where $\lambda_j \in \Q$ are such that $m\lambda_j \in \Z$ and $\lambda_j \in [0,1)$ (so that $\lambda_j \in \{0,1/m,2/m,\dots,(m-1)/m\}$). We then define a
map $\age: S_\G \to \Q$ by mapping $g \in S_\G$ to
\[
\age(g)=\sum_{j=1}^n \lambda_j.
\]
Since eigenvalues do not change under conjugation, it is invariant with respect to the $\G$-action on $S_\G$ and thus induces a map (still called $\age$) $\age: |\Lambda\G| \to \Q$. By continuity of eigenvalues, the map is continuous and since it takes values in $\Q$ it must be locally constant. Hence, it is constant in each of the (un)twisted sectors, so we denote by $\age{(g)}$ the value that it takes in the (un)twisted sector $|\G_{(g)}|$. In particular, it is clear that $\age{(1)} = 0$.

\begin{definition}
Let $R$ be a field of characteristic zero. The \emph{Chen-Ruan cohomology} of an orbifold $\G$ with coefficients in $R$ is the $\Q$-graded vector space
\[
\H^*_{CR}(\G;R) = \bigoplus_{(g) \in T} \H^{*-2\age{(g)}}(|\G_{(g)}|;R).
\]
\end{definition}

\begin{example}
If $G \subset \U(n)$ is a subgroup that acts freely in $\C^n\setminus\{0\}$ then clearly
\begin{equation}
\label{eq:crc-C^n}
\H^*_{CR}(\C^n/G;\Q) = \bigoplus_{(g) \in \text{Conj}(G)}\Q[-2\age(g)],
\end{equation}
where $-2\age(g)$ is the shift in the grading.
\end{example}

\subsection{Symplectic homology/cohomology of orbifold fillings}
\label{sec:sh-orb}

Let $(M^{2n+1},\xi)$ be a closed smooth contact manifold endowed with an exact orbifold symplectic filling $W$ (see \cite{GZ} for a definition of an exact orbifold filling). Suppose that the rational first Chern class $c_1^\Q(TW)$ vanishes. Gironella and Zhou introduced in \cite{GZ} a symplectic cohomology of $W$ that will play a fundamental role in this work. Differently from the grading in symplectic homology established in Section \ref{sec:SH} (given by plus the Conley-Zehnder index), the grading of $\SH^*(W)$ in \cite{GZ} is given by \emph{$n+1$ minus the Conley-Zehnder index} (note here that the dimension of the filling $W$ is $2n+2$ and in \cite{GZ} it has dimension $2n$).

The orbifold symplectic cohomology is important for us because when $M$ is a lens space $\L$ its natural symplectic filling $W=\C^{n+1}/\Z_p$ is an exact orbifold. Since the lens space is smooth, we can consider the symplectic homology $\SH_*(M)$ of its symplectization but it is defined only for special contact forms (more precisely, for contact forms such that every contractible periodic orbit $\ga$ satisfies $\cz(\ga)>3-n$). However, the orbifold symplectic cohomology $\SH^*(W)$ is defined for any contact form. Moreover, the computation of $\SH^*(W)$ is useful for the computation of the shift operator in the exact triangle \eqref{eq:et gysin symplectization} as we will see in the proof of Theorem \ref{thm:multiplicity}.

As proved in \cite[Theorem A]{GZ} the tautological exact triangle \eqref{eq:etsh} also holds for the symplectic cohomology of orbifolds, but now $\SH_-^*(W)$ is isomorphic to the Chen-Ruan cohomology $\H^*_{CR}(W)$, where from now on we are taking rational coefficients (note the choice of the grading $n+1$ minus the Conley-Zehnder index in $\SH^*(W)$). Therefore, we have the exact triangle
\begin{equation}
\label{eq:etsh-orbifold-GZ}
\xymatrix{
		\H^*_{CR}(W) \ar[rr] & &\SH^*(W)\ar[ld] \\
		& \SH^*_+(W) \ar[lu]^{[1]} &}.
\end{equation}

The symplectic cohomology of orbifolds is defined, as in the smooth case, as a limit of Floer cohomologies of a suitable class of Hamiltonians on $W$. The periodic orbits and Floer trajectories are orbifold maps $\ga: S^1 \to \W$ and $u: \R \times S^1 \to \W$ respectively, where $\W$ is the completion of $W$. Suppose, for instance, that $W$ is a global quotient orbifold $[X/G]$. Then the periodic orbits and Floer trajectories correspond to smooth maps $\ga: [0,1] \to X$ and $u: \R \times [0,1] \to X$ satisfying $g$-boundary conditions for $g \in G$; cf. \cite{Mor}.

Similarly, we can define the symplectic homology $\SH_*(W)$ of exact orbifold fillings $W$ as in \cite{GG}, using the sign conventions and the grading of \cite{GG}. It turns out that $\SH_*(W)$ and $\SH^*(W)$ are the same theory. More precisely,
\begin{equation}
\label{eq:homology x cohomology}
\SH^{(n+1)-*}(W) = \SH_*(W). 
\end{equation}
Indeed, the differential of Floer cohomology in \cite{GZ} from a periodic orbit $\ga_0$ to $\ga_1$ counts solutions $u$ of Floer's equation (with the convention of Floer's equation in \cite{GZ}) such that $\lim_{s\to\infty} u(s,\cdot)=\ga_0(\cdot)$ and $\lim_{s\to-\infty} u(s,\cdot)=\ga_1(\cdot)$. On the other hand, the differential in Floer homology in \cite{GG}  from a periodic orbit $\ga_0$ to $\ga_1$ counts solutions $u$ of Floer's equation (with the convention of Floer's equation in \cite{GG} which coincides with our convention) such that $\lim_{s\to-\infty} u(s,\cdot)=\ga_0(\cdot)$ and $\lim_{s\to\infty} u(s,\cdot)=\ga_1(\cdot)$. But, as explained in Section \ref{sec:signs}, the solutions $u(s,\cdot)$ of Floer's equation in \cite{GG} correspond to solutions $u(-s,\cdot)$ of Floer's equation in \cite{GZ}. Hence, the differentials coincide and consequently we conclude \eqref{eq:homology x cohomology}. Note that the Fredholm index of the linearized Floer operator in \cite{GG} is $\cz(\ga_0)-\cz(\ga_1)$ and therefore the differential in \cite{GG} decreases the degree. On the other hand, the Fredholm index in \cite{GZ} is $\cz(\ga_0)-\cz(\ga_1)=((n+1)-\cz(\ga_1))-((n+1)-\cz(\ga_0))$ and consequently the differential in \cite{GZ} increases the degree (given by $n+1$ minus the Conley-Zehnder index).

Thus, the previous exact triangle \eqref{eq:etsh-orbifold-GZ} is equivalent to the exact triangle
\begin{equation}
\label{eq:etsh-orbifold}
\xymatrix{
		\H^{(n+1)-*}_{CR}(W) \ar[rr] & &\SH_*(W)\ar[ld] \\
		& \SH_*^+(W) \ar[lu]^{[-1]} &}.
\end{equation}

A word is due to the left upper map. Given $k \in \Q$ this map sends $\SH_k^+(W)$ to $\H^{(n+1)-(k-1)}_{CR}(W)$ so it increases the degree in Chen-Ruan cohomology. The reason why is that, by equality \eqref{eq:homology x cohomology}, when we decrease the degree in symplectic homology we increase the degree in symplectic cohomology. More precisely, the exact triangle \eqref{eq:etsh-orbifold} is obtained in the following way: we have the tautological exact sequence
\[
\cdots \to \SH^+_k(W) \to \SH^-_{k-1}(W) \to \SH_{k-1}(W) \to \SH^+_{k-1}(W) \to \cdots
\]
where $\SH^-_*(W)$ is the negative symplectic homology. By \eqref{eq:homology x cohomology},
\[
\SH^-_{k-1}(W) = \SH^{(n+1)-(k-1)}_-(W) \cong \H_{CR}^{(n+1)-(k-1)}(W),
\]
where $\SH^*_-$ is the negative symplectic cohomology and the last isomorphism follows from \cite{GZ}. Therefore the last sequence yields
\[
\cdots \to \SH^+_k(W) \to \H_{CR}^{(n+1)-(k-1)}(W) \to \SH_{k-1}(W) \to \SH^+_{k-1}(W) \to \cdots
\]
which is the exact triangle \eqref{eq:etsh-orbifold}.

The symplectic homology groups have a decomposition given by the free homotopy classes of $W$ that correspond to the conjugacy classes of $\pi_1^{\text{orb}}(W)$. Consider the case that $W=\C^{n+1}/G$, where $G \subset \U(n)$ is a finite subgroup. Then $\pi_1^{\text{orb}}(W) \cong G$ and, by definition, the Chen-Ruan cohomology $\H^*_{CR}(W)$ has a decomposition indexed by the conjugacy classes of $G$. It is easy to see from the proof of \cite[Theorem A]{GZ} that the exact triangle \eqref{eq:etsh-orbifold} respects this decomposition.

\section{Proof of Theorem \ref{thm:existence}}
\label{sec:proof existence}

Let $\beta$ be the induced contact form on $\L$. First we need the following easy result. Note that given a closed Reeb orbit $\ga: S^1=\R/\Z \to M$ and $c \in S^1$ we have another periodic orbit given by $\ga(t+c)$. We refer to an equivalence class of parametrized closed orbits under reparametrization by $S^1$ as an \emph{unparametrized} closed orbit.

\begin{proposition}
\label{prop:symmetric orbit}
The simple symmetric unparametrized closed orbits of $\alpha$ are in bijection with the simple unparametrized closed orbits of $\beta$ whose homotopy classes are generators of $\pi_1(\L)$.
\end{proposition}

\begin{proof}
Let $A$ be the set of simple symmetric unparametrized closed orbits of $\alpha$ and $B$ the set of simple unparametrized closed orbits of $\beta$ whose homotopy classes are generators of $\pi_1(\L)$. Given $\ga \in A$ and $x \in \Im(\ga)$, note that the symmetry of $\ga$ implies that the $\Z_p$-orbit of $x$ is contained in $\Im(\ga)$. Take an isomorphism between $\pi_1(\L)$ and the group of deck transformations on $S^{2n+1}$ and given $a \in \pi_1(\L)$ let $\psi_a$ be the corresponding map.

Let $\pi: S^{2n+1} \to \L$ be the quotient projection. Given $\ga \in A$ define $\phi(\ga)$ as the simple unparametrized closed orbit whose image is $\pi \circ \ga$. We have that the homotopy class $[\phi(\ga)]$ is a generator of $\pi_1(\L)$ since, otherwise, the $\Z_p$-orbit of $x$ would not be contained in $\Im(\ga)$. Thus, $\phi$ defines a map $\phi: A \to B$.

We have that $\phi$ is injective: clearly, given two distinct $\ga_1,\ga_2 \in A$ we have that $\phi(\ga_1) \neq \phi(\ga_2)$. To show that $\phi$ is onto, we argue as follows: given $\tga \in B$ let $\tga_0: [0,T] \to \L$ be a parametrized simple closed orbit of $\beta$ such that $\tga_0(T)=\tga_0(0)$ and $\Im(\tga_0)=\Im(\tga)$. Let $\bga_0$ be a lift of $\tga_0$ to $S^{2n+1}$. Let $x=\bga_0(0)$ so that $\bga_0(T)=\psi_a(x)$ where $a=[\tga]$. Define the closed curve $\ga_0: [0,pT] \to S^{2n+1}$ as
\[
\ga_0(t)=\psi_a^k\circ\bga_0(t-kT)\text{ if } t \in [kT,(k+1)T]\text{ for }k \in \{0,\dots,p-1\}.
\]
Clearly, by the symmetry of $\alpha$ and the fact that $a$ is a generator of $\pi_1(\L)$, $\ga_0$ is a symmetric simple closed orbit of $\alpha$. Let $\ga$ be the corresponding unparametrized closed orbit. By construction, $\phi(\ga)=\tga$ and therefore we conclude that $\phi$ is surjective.
\end{proof}

Thus, in order to prove Theorem \ref{thm:existence}, we have to show that given any $a \in \pi_1(\L)$ there exists a closed orbit of $\beta$ with homotopy class $a$. Consider the discussion in Section \ref{sec:sh-orb}. Suppose that $M=\L$ and $W=\C^{n+1}/\Z_p$ where the $\Z_p$-action is the one generated by \eqref{eq:Z_p-action}. Note that $c_1^\Q(TW)=0$. Then, by the definition of the Chen-Ruan cohomology, $\H^*_{CR}(W)$ has a decomposition indexed by $(g)=g \in \Z_p$. Moreover, $\SH_*(W)$ and $\SH_*^+(W)$ also have decompositions indexed by $\pi_1^{\text{orb}}(W) \cong \Z_p$. Denote these decompositions by
\[
\H^*_{CR}(W) = \bigoplus_{k\in \Z_p} \H^*_{CR,k}(W),\ \SH_*(W) = \bigoplus_{k\in \Z_p} \SH_*^k(W)\text{ and }\SH_*^+(W) = \bigoplus_{k\in \Z_p} \SH_*^{k,+}(W).
\]
As mentioned in Section \ref{sec:sh-orb}, an inspection of the proof of \cite[Theorem A]{GZ} shows that the exact triangle \eqref{eq:etsh-orbifold} respects these decompositions.

Recall that we are taking all these homology groups with rational coefficients. Thus, by \cite[Theorem B]{GZ} and \eqref{eq:homology x cohomology}, $\SH_*(W)=0$ and therefore, by the exact triangle \eqref{eq:etsh-orbifold}, $\H^{(n+1)-(*-1)}_{CR}(W) \cong \SH_*^+(W)$ respecting the aforementioned decompositions (possibly up to a permutation in the indexes of the decomposition according to the choice of the isomorphism $\pi_1^{\text{orb}}(W) \cong \Z_p$). By \eqref{eq:crc-C^n}, given $k \in \{0,\dots,p-1\} \cong \Z_p$ we have that
\begin{equation}
\label{eq:CR filling}
\H^*_{CR,k}(W) =
\begin{cases}
\Q\quad\text{if }*=2\sum_{i=0}^{n+1} \{k\ell_i/p\}\\
0\quad\text{otherwise.}
\end{cases}
\end{equation}
Thus, $\SH_{*}^{k,+}(W)\neq 0$ for every $k$. By the definition of $\SH_*^{k,+}(W)$, it immediately implies that given any $a \in \pi_1(\L)$ we have that $\beta$ has a closed orbit with homotopy class $a$.

\section{Proof of Theorem \ref{thm:multiplicity}}
\label{sec:proof multiplicity}

Let $\beta$ be the induced contact form on $\L$. By Proposition \ref{prop:symmetric orbit}, we have to show that given any $a \in \pi_1(\L)$ there exist two distinct $a$-simple closed orbits of $\beta$ (recall that a periodic orbit is $a$-simple if it has homotopy class $a$ and it is not an iterate of another periodic orbit with homotopy class $a$, although it can be the iterate of a closed orbit with another homotopy class). Assume, from now on, that $\beta$ has finitely many $a$-simple periodic orbits  (otherwise, there is nothing to prove) and let $\{\bga_1,\dots,\bga_m\}$ be the set of $a$-simple closed orbits.  By Theorem \ref{thm:existence}, $m\geq 1$. We have to show that $m\geq 2$. As in the previous section, let $M=\L$ and $W=\C^{n+1}/\Z_p$ be its orbifold filling, where the $\Z_p$-action is the one generated by \eqref{eq:Z_p-action}. By \cite[Theorem B]{GZ} and \eqref{eq:homology x cohomology}, $\SH_*(W)=0$ (recall that we are taking rational coefficients) and therefore, by the exact triangle \eqref{eq:etsh-orbifold}, $\H^{(n+1)-(*-1)}_{CR}(W) \cong \SH_*^+(W)$. In particular, by \eqref{eq:CR filling}, we have that $\SH_k^+(W)=0$ for every $k>n+2$.

The following result is well known to experts in the area. For the sake of completeness, we will provide a sketch of the proof.

\begin{proposition}
\label{prop:iso}
Let $(M^{2n+1},\xi)$ be a smooth closed contact manifold endowed with an exact orbifold symplectic filling $W$ such that $c_1^\Q(TW)=0$. Suppose that $c_1(\xi)|_{H_2(M,\R)}=0$ and that $M$ admits a non-degenerate contact form $\alpha$ supporting $\xi$ such that every closed Reeb orbit $\ga$ contractible in $W$ satisfies $\cz(\ga)>3-n$. Then we have an isomorphism between $\SH_*(M)$ and $\SH^+_*(W)$ respecting the grading.
\end{proposition}

\begin{proof}
The proof goes as the proof of \cite[Proposition 9.17]{CO} (\cite[Proposition 9.17]{CO} is for smooth fillings but it works equally well for exact orbifold fillings). It is based on a stretch of the neck argument already used in \cite{BO09a,CFO} and the idea is the following; the reader can find details in \cite{CO}. Consider an admissible Hamiltonian $H: S^1 \times \W \to \R$, where $\W$ is the completion of $W$, used to define $\SH^+_*(W)$ (take, for instance, the type (II) of Hamiltonians considered in \cite[Section 3.1]{GZ}, where $H$ is $C^2$-small and autonomous on $W$ and $H$ depends only on the radius on $\W\setminus W$). Consider the Floer trajectories joining periodic Reeb orbits (which are located in the symplectization part $\W\setminus W$). Note that these are the Floer trajectories relevant in the definition of $\SH^+_*(W)$.

Now, degenerate the almost complex structure in stretch-of-the-neck manner near $\partial W$; see \cite{BO09a,CFO,CO}. We claim that, under the assumption on the indices, for a sufficiently stretched neck all Floer trajectories stay in the symplectization part and thus coincide with the Floer trajectories involved in the version of symplectic homology $\SH_*(M)$ defined using only the symplectization.

Arguing by contradiction, suppose that we would see in the limit a building whose top component, in the symplectization, is a curve which contains both punctures of the initial Floer trajectory and solves a Cauchy-Riemann equation with Hamiltonian perturbation. That this top component curve contains both Floer punctures is a consequence of an asymptotic behavior lemma developed in \cite[Pages 654-655]{BO09a} and \cite[Lemma 2.3]{CO}. By our hypothesis of contradiction, it must have at least one negative puncture to which is attached a holomorphic building which end up eventually in the filling; this implies that the orbit in this negative puncture must be contractible in $W$.

This top component is regular because it solves a perturbed equation near the Floer punctures. But it has a positive Floer puncture with SFT degree $k$ (the Symplectic Field Theory (SFT) degree is given by the Conley-Zehnder index plus $n-2$), a negative Floer puncture with SFT degree $k-1$ (these punctures are the punctures of the original Floer trajectory) and at least one other negative puncture with SFT degree strictly bigger than one by our index assumptions (note that the periodic orbit in this puncture must be contractible in $W$). Therefore, the Fredholm index of this top component curve, given by the SFT degree of the positive puncture minus the sum of the SFT degrees of the negative punctures, must be negative, contradicting the regularity. Such a curve therefore cannot exist. 

Finally, note here that, since the region where we stretch the neck is away from the singularities of $W$ as well as the top component, the singular part of $W$ does not play any role in this argument.
\end{proof}

Clearly, the lens space $M=\L$ with its filling $W=\C^{n+1}/\Z_p$ satisfies the assumptions of the previous proposition. (Note here that $\pi_1^{\text{orb}}(W) \cong \pi_1(M) \cong \Z_p$ and therefore a non-degenerate dynamically convex contact form on $M$ has the property that every periodic orbit $\ga$ contractible in $W$ satisfies $\cz(\ga)>3-n$.) Now, we need the following result which has independent interest.

\begin{proposition}
\label{prop:HC}
We have that
\begin{equation}
\label{eq:HC lens space}
\HC_*^{a}(M) =
\begin{cases}
\Q\quad\text{if }*=k_a + 2k\text{ for every }k \in \N_0\\
0\quad\text{otherwise,}
\end{cases}
\end{equation}
where $k_a=\min\{k \in \Q;\, \HC_k^a(M)\neq 0\}$. 
\end{proposition}

\begin{remark}
The grade $k_a$ is finite and can be easily computed from the weights of the action and $p$; see \cite[Equation (1.2)]{AM}. Note that for our global quotient orbifold $W$ we can naturally identify $S_\G/\simeq$, where $x \simeq y$ if and only if $x$ and $y$ belong to the same connected component of $S_\G$, with $\pi_1^{\text{orb}}(W) \cong \pi_1(M)$ and consequently we can see the age map, defined in Section \ref{sec:CR}, as a function $\age: \pi_1(M) \to \Q$. Using \cite[Equation (1.2)]{AM} we can easily derive the following computation of $k_a$ in terms of the age:
\begin{equation*}
k_a=
\begin{cases}
2\,\age(a)-n\quad\text{if }a\neq 0\\
n+2\quad\text{otherwise.}
\end{cases}
\end{equation*}
\end{remark}

\begin{proof}
As mentioned in Section \ref{sec:SH}, $\HC_\ast (M)$ can be described as the homology of a complex $\CC_\ast (\alpha)$ generated by good closed Reeb orbits, for some non-degenerate contact form $\alpha$, graded by the Conley-Zehnder index. For lens spaces, which are very special  toric contact manifolds, we can use toric contact geometry to explicitly write down non-degenerate toric Reeb vector fields with finitely many simple closed Reeb orbits and explicitly compute the Conley-Zehnder indices of all of them and of all of their iterates. For lens spaces with zero first Chern class, i.e. Gorenstein contact lens spaces, this has been done with all details in Section 5 of \cite{AMM-MRL} with the following outcome:
\begin{itemize}
\item With an appropriate choice of toric contact form $\alpha$ and up to an arbitrarily large degree, $\CC_\ast (M)$ is generated by a single good simple closed Reeb orbit $\gamma$ and its iterates.
\item $\gamma$ is a generator of the fundamental group $\pi_1 (M)$ and 
\[
\cz \left( \gamma^{N+p} \right) = \cz \left( \gamma^{N} \right) + 2 
\quad \text{up to an arbitrarily large $N\in\N$ ($p=|\pi_1 (M)|$).}
\]
\item It follows that $\HC_\ast^a (M)$ satisfies the result stated in this proposition for any free homotopy class $a$.
\end{itemize}
Using the Conley-Zehnder index formula obtained in Section 4 of~\cite{AMM-Q} and valid in the context of $\Q$-Gorenstein toric contact manifolds, i.e. toric contact manifolds with torsion first Chern class, the same approach can be adapted to general lens spaces as we will now explain.

Let $M = L^{2n+1}_p (\ell_0, \ldots, \ell_n)$ be an arbitrary lens space and assume, without any loss of generality, that $\ell_n = 1$. As a toric contact manifold, $M$ is determined by a moment cone $C\subset\R^{n+1}$ whose facets have defining normals of the form
\[
\nu_0=\begin{bmatrix}
0\\ 
\vdots\\ 
0\\ 
k\\ 
m
\end{bmatrix},\,\,
 \nu_j=\begin{bmatrix}
\ \\ 
e_j\\ 
\ \\ 
0 \\
m
\end{bmatrix}\textup{ for }j=1, \ldots, n-1,\,\, \nu_n=\begin{bmatrix}
-\ell_1\\ 
\vdots\\ 
-\ell_{n-1}\\ 
q\\ 
m
\end{bmatrix}\,,
\]
where
\begin{itemize}
\item $\left\{e_1, \ldots, e_{n-1}  \right\}$ is the canonical basis of $\Z^{n-1}$;
\item $m\in\left\{ 1, \ldots, p-1 \right\}$ is the minimal positive integer such that
\[
m \cdot c_1 (TM) = 0 \in H^2 (M; \Z) \cong \Z_p \quad\text{(note that $m$ divides $p$);}
\]
\item $k\in\left\{ 0, \ldots, p-1 \right\}$ is the minimal non-negative integer such that
\[
k \cdot \left( \ell_0 + \ell_1 + \cdots + \ell_n  \right) \equiv  \frac{p}{m} \mod p
\]
(note that $c_1(M) \equiv \ell_0 + \ell_1 + \cdots + \ell_n \mod p$);
\item $q = k \cdot \left( \ell_1 + \cdots + \ell_n \right) - p/m$.
\end{itemize}
In fact, the matrix $\beta = [\nu_0 | \nu_1 | \cdots | \nu_n ]$ is such that
\begin{align*}
\det (\beta) & = \pm \left(km\left(\ell_1 +\cdots+\ell_{n-1}+1\right) - mq\right)  \\
& = \pm m \left(k \left( \ell_1 + \cdots + \ell_n \right) - k \left( \ell_1 + \cdots + \ell_n \right) + \frac{p}{m}\right) \\
& = \pm p
\end{align*}
and
\[
\ker \left(\beta:\T^{n+1} \to \T^{n+1}  \right)
\]
is generated by
\[
\left( -\frac{q}{p} , \frac{k\ell_1}{p}, \ldots, \frac{k\ell_{n-1}}{p}, \frac{k}{p}  \right)\,.
\]
Hence, the toric contact manifold determined by this moment cone $C\subset \R^{n+1}$ is
\[
L^{2n+1}_p (-q, k\ell_1, \ldots, k\ell_{n-1}, k).
\]
Since
\begin{align*}
-q & = -k \left( \ell_1 + \cdots + \ell_n \right) + \frac{p}{m} \\
& \equiv -k \left( \ell_1 + \cdots + \ell_n \right) + k \left( \ell_0 + \ell_1 + \cdots + \ell_n \right) \mod p \\
& = k \ell_0
\end{align*}
we have that
\begin{align*}
L^{2n+1}_p (-q, k\ell_1, \ldots, k\ell_{n-1}, k) & = L^{2n+1}_p (k \ell_0, k\ell_1, \ldots, k\ell_{n-1}, k) \\
& = L^{2n+1}_p ( \ell_0, \ell_1, \ldots, \ell_{n-1}, 1) \\
& = M \,.
\end{align*}

Any positive linear combination of the normals $\nu_0, \nu_1, \ldots, \nu_n$, gives rise to a toric Reeb vector field on $M$. In particular, we can consider toric Reeb vector fields of the form
\[
R = \sum_{j=0}^{n-1} \veps_j \nu_j  + \left( 1 - \veps \right) \nu_n
\]
for arbitrarily small $\veps_0, \ldots, \veps_{n-1} > 0$ and $\veps = \sum_{j=0}^{n-1} \veps_j$. By choosing $\veps_0, \ldots, \veps_{n-1}$ such that $\{ \veps_0, \ldots, \veps_{n-1} , 1\}$ are $\Q$-independent, we know that the Reeb flow of $R$ has exactly $n+1$ simple closed orbits, one for each edge of the moment cone $C\subset\R^{n+1}$, and these are all non-degenerate and generators of the fundamental group $\pi_1 (M)$ (see \cite[Section 2.2]{AMM-MRL}). Moreover, since $R\approx \nu_n$, the $n$ orbits corresponding to the $n$ edges contained in the facet of $C$ with normal $\nu_n$, and all their iterates, have an arbitrarily large Conley-Zehnder index (for arbitrarily small $\veps_0, \ldots, \veps_{n-1} > 0$). Hence, $\HC_*(M)$ can be completely determined by computing
\[
\cz \left( \gamma^N \right) \ \text{for all $N\in\N$ and arbitrarily small $\veps_0, \ldots, \veps_{n-1} > 0$,}
\]
where $\gamma$ is the simple closed $R$-orbit corresponding to the edge of $C$ determined by the $n$ facets with normals $\nu_0, \ldots, \nu_{n-1}$. 

This Conley-Zehnder index computation can be done in the following way:
\begin{itemize}
\item[(i)] Complete $\{\nu_0, \ldots, \nu_{n-1} \}$ to a $\Z$-basis $\{\nu_0, \ldots, \nu_{n-1}, \eta \}$ of $\Z^{n+1}$ with the vector
\[
\eta = (0, \ldots, 0, c, d)^t\,, \ \text{where $c,d\in\Z$ are such that $kd - mc = 1$.}
\]
A simple computation shows that
\[
\nu_n = a_0 \nu_0 + \sum_{j=1}^{n-1} (-\ell_j) \nu_j  + p \eta \,,\ \text{with}\ 
a_0 =  \sum_{j=1}^{n} \ell_j  - d \frac{p}{m}\,.
\]
\item[(ii)] Write $R$ in the basis $\{\nu_0, \ldots, \nu_{n-1}, \eta \}$:
\begin{align*}
R & =  \sum_{j=0}^{n-1} \veps_j \nu_j  + \left( 1 - \veps \right) \nu_n \\
& = \sum_{j=0}^{n-1} \veps_j \nu_j  + 
\left( 1 - \veps \right) \left( a_0 \nu_0 + \sum_{j=1}^{n-1} (-\ell_j) \nu_j  + p \eta \right) \\
& = (\veps_0 + (1-\veps) a_0) \nu_0 + \sum_{j=1}^{n-1} (\veps_j -(1-\veps)\ell_j) \nu_j 
+ (1-\veps) p \eta \,.
\end{align*}
\item[(iii)] Use the Conley-Zehnder index formula obtained in Section 4 of~\cite{AMM-Q} to find that
\begin{equation} \label{eq:CZindex}
\cz \left( \gamma^N \right) = 2 \left( \left\lfloor \frac{N}{p} \left(a_0 + \frac{\veps_0}{1-\veps}    \right)\right\rfloor 
+ \sum_{j=1}^{n-1}  \left\lfloor \frac{N}{p} \left(-\ell_j + \frac{\veps_j}{1-\veps}    \right)\right\rfloor  
+ N \frac{d}{m} \right) + n \,.
\end{equation}
\end{itemize}

We can now complete the proof of the proposition in the following way. Let $D\in\N$ be an arbitrarily high degree. It follows from~(\ref{eq:CZindex}) that
\[
\cz \left( \gamma^N \right) \geq \frac{2N}{p} -n\,, \ \forall N\in\N\,,
\]
and so
\[
\cz \left( \gamma^N \right) > D \,, \ \forall N > \frac{p(D+n)}{2}\,.
\]
Choose $\veps_0, \ldots, \veps_{n-1} > 0$ sufficiently small so that:
\begin{itemize}
\item[(i)] the $n$ orbits corresponding to the $n$ edges contained in the facet of C with normal $\nu_n$, and all their iterates, have Conley-Zehnder index greater than $D$;
\item[(ii)] for all $N\leq p(D+n)/2$ we have that
\begin{align*}
\cz \left( \gamma^N \right) & = 2 \left( \left\lfloor \frac{N}{p} \left(a_0 + \frac{\veps_0}{1-\veps}    \right)\right\rfloor 
+ \sum_{j=1}^{n-1}  \left\lfloor \frac{N}{p} \left(-\ell_j + \frac{\veps_j}{1-\veps}    \right)\right\rfloor  
+ N \frac{d}{m} \right) + n \\
& = 2 \left( \left\lfloor \frac{N a_0}{p}\right\rfloor 
+ \sum_{j=1}^{n-1}   \left\lfloor \frac{-N \ell_j}{p}\right\rfloor  
+ N \frac{d}{m} \right) + n\,.
\end{align*}
\end{itemize}
We can now finish the proof of this proposition by showing that
\[
\cz \left(\gamma^{N+p}\right) = \cz \left(\gamma^{N}\right) + 2 \quad \text{for all} \  N\leq \frac{p(D+n)}{2} - p\,.
\]
In fact,
\begin{align*}
\cz \left(\gamma^{N+p}\right) & = 2 \left( \left\lfloor \frac{N a_0}{p} + a_0\right\rfloor 
+ \sum_{j=1}^{n-1}   \left\lfloor \frac{-N \ell_j}{p} - \ell_j \right\rfloor  
+ N \frac{d}{m} + p \frac{d}{m} \right) + n \\
& = \cz \left( \gamma^N \right)  + 2 \left( a_0 + \sum_{j=1}^{n-1} (-\ell_j) + \frac{pd}{m}\right) \\
& = \cz \left( \gamma^N \right)  + 2 \left( \sum_{j=1}^{n} \ell_j - \frac{dp}{m} + \sum_{j=1}^{n-1} (-\ell_j) + \frac{pd}{m}\right) \\
& = \cz \left( \gamma^N \right) + 2 (\ell_n) \\
& = \cz \left( \gamma^N \right) + 2 \quad \text{for all} \  N\leq \frac{p(D+n)}{2} - p\,.
\end{align*}
\end{proof}

From the previous two propositions we conclude the following result.

\begin{proposition}
\label{prop:shift operator}
Let $a \in \pi_1(M)$ and set $l_a=\min\{k_a + 2k > n+2;\,k\in \N_0\}$, with $k_a$ as in the last proposition. Then, for every $k\in \N_0$, $\HC^a_{l_a+2k}(M)=\Q$ and $D: \HC_{l_a+2k}^a(M) \to \HC_{l_a+2k-2}^a(M)$ is an isomorphism.
\end{proposition}

\begin{proof}
The fact that $\HC^a_{l_a+2k}(M)=\Q$ for every $k\in \N_0$ follows from Proposition \ref{prop:HC}. By Proposition \ref{prop:iso}, we have that $\SH_*(M) \cong \SH_*^+(W)$ and, therefore, $\SH_k(M)=0$ for every $k>n+2$, because, as explained at the end of the previous section, $\SH_*^+(W) \cong \H^{(n+1)-(*-1)}_{CR}(W)$ and $\H^k_{CR}(W)=0$ for every $k<0$. Thus, by the exact triangle \eqref{eq:et gysin symplectization}, we have that $D: \HC_{k}(M) \to \HC_{k-2}(M)$ is an isomorphism for every $k>n+2$. As discussed before, $D$ respects the decomposition by homotopy classes and consequently $D: \HC_k^a(M) \to \HC_{k-2}^a(M)$ is an isomorphism for every $k>n+2$. Hence, by the definition of $l_a$, we conclude that $D: \HC_{l_a+2k}^a(M) \to \HC_{l_a+2k-2}^a(M)$ is an isomorphism for every $k\in \N_0$.
\end{proof}

From the previous proposition and the discussion in Section \ref{sec:LS}, we have an injective map $\psi: \N_0 \to \PP_a$, where $\PP_a$ is the set of closed Reeb orbits with homotopy class $a$, such that $\ga_k:=\psi(k)$ satisfies
\begin{equation}
\label{eq:index}
|\cz(\ga_k)-(l_a+2k-2)|\leq 2n.
\end{equation}
(Note here that the periodic orbits with homotopy class $a$ are isolated since we are assuming that there exist finitely many $a$-simple orbits.) Consider the sequence of numbers
\[
x(k)=\cz(\ga_k).
\]
By \eqref{eq:index}, its density $\delta(x):=\lim_{j\to\infty}\frac{1}{j}\#\{i; x(i) \leq j\}$ equals $1/2$.

Recall that $\{\bga_1,\dots,\bga_m\}$ is the set of $a$-simple periodic orbits of the contact form $\beta$. Consider the sequences of numbers
\[
\bx_j(i) = \cz(\bga_j^{ip+1})
\]
for $j \in \{1,\dots,m\}$ and $i \in \N_0$. Since $|\cz(\bga_j^{ip+1})-(ip+1)\Delta(\bga_j)|\leq n$ for every $i \in \N_0$, we have that $\delta(\bx_j)=1/(p\Delta(\bga_j))$, where $\Delta(\bga_j)=\lim_{k\to\infty} \frac 1k \cz(\bga_j^k)$ is the mean index of $\bga_j$.

Now, note that each point in the sequence $x$ belongs to one of the sequences $\bx_j$ and no point in the sequences $\bx_j$ can be used twice, because $\psi$ is injective. So the density of $x$ must be no more than the sum of the densities of $\bx_j$:
\[
\delta(x) \leq \sum_{j=1}^m \delta(\bx_j)\quad\Longleftrightarrow\quad \frac 12 \leq \sum_{j=1}^m \frac{1}{p\Delta(\bga_j)},
\]
that is,
\begin{equation}
\label{eq:final ineq}
\frac p2 \leq \sum_{j=1}^m \frac{1}{\Delta(\bga_j)}.
\end{equation}
Now, suppose that $m=1$, that is, there exists only one $a$-simple closed orbit $\ga$. Since $\alpha$ is dynamically convex, $\Delta(\ga^p)>2$ (see \cite[Corollary 5.1]{Lon00}) which implies that $\Delta(\ga)>2/p$. Thus,
\[
\frac{1}{\Delta(\ga)} < \frac p2.
\]
But, by \eqref{eq:final ineq},
\[
\frac{1}{\Delta(\ga)} \geq \frac p2,
\]
a contradiction.

\begin{remark}
The assumption of dynamical convexity of $\alpha$ is used to ensure that $\Delta(\ga^p)>2$ and that every contractible periodic orbit of $\beta$ has index bigger than $3-n$ so that we can define the equivariant symplectic homology of $\beta$ in the symplectization of $\L$.
\end{remark}

\end{document}